\documentclass[12pt,letterpaper]{amsart}
\usepackage{graphicx}
\usepackage[utf8]{inputenc}
\usepackage{amssymb}
\usepackage{epstopdf}
\usepackage{amsmath}
\usepackage{xcolor}

\newtheorem{corollary}{Corollary}[section]
\newtheorem{theorem}{Theorem}[section]
\newtheorem{lemma}{Lemma}[section]
\newtheorem{remark}{Remark}[section]

\newtheorem{definition}{Definition}[section]

\newcommand{\abs}[1]{\left\lvert #1 \right\rvert}

\def\R{{\mathbb R}}
\def\C{{\mathbb C}}
\def\Z{{\mathbb Z}}

\def\bc{\begin{center}}
\def\ec{\end{center}}
\def\be{\begin{equation}}
\def\ee{\end{equation}}
\def\ba{\begin{array}}
\def\ea{\end{array}}
\def\bea{\begin{eqnarray}}
\def\eea{\end{eqnarray}}
\def\beaa{\begin{eqnarray*}}
\def\eeaa{\end{eqnarray*}}

\def\ifl{\iffalse}

\begin{document}

\title[]{Weak separability and partial Fermi isospectrality of discrete periodic Schr\"odinger operators}
\author[]
{Jifeng Chu$^{1}$, \quad Kang Lyu$^2$, \quad Chuan-Fu Yang$^2$}

\address{$^1$ School of Mathematics, Hangzhou Normal University, Hangzhou 311121, China}
\address{$^2$ School of Mathematics and Statistics, Nanjing University of Science and Technology, Nanjing 210094, China}
\email{jifengchu@126.com (J. Chu)}
\email{lvkang201905@outlook.com (K. Lyu)}
\email{chuanfuyang@njust.edu.cn (C. Yang)}
\thanks{Jifeng Chu was supported by the National Natural Science Foundation of China (Grant
No. 12571168) and by the Science and Technology Innovation Plan of Shanghai (Grant No. 23JC1403200).}



\begin{abstract}
In this paper, we consider the discrete periodic Schr\"odinger operators $\Delta+V$ on $\Z^d$, where
$V$ is  $\Gamma$-periodic with $\Gamma=q_1 \mathbb{Z}\oplus q_2\mathbb{Z}\oplus\cdots\oplus q_d\mathbb{Z}$ and
positive integers $q_j$, $j=1,2,\cdots,d,$ are pairwise coprime.
{We introduce the notions of generalized partial Fermi isospectrality and weak separability, and prove that two {generalized} partially Fermi isospectral potentials have the same weak separability. As a direct application, we can prove that two potentials have the same $(d_1,d_2,\cdots,d_r)$-separability by assuming that they are generalized partially Fermi isospectral,
	  instead of the  Fermi isospectrality or Floquet isospectrality. Besides, we prove that each couples of components of the generalized Fermi isospectral potentials are Floquet isospectral in some sense.}
\end{abstract}

\maketitle

\section{Introduction}
In the past several decades, the study of inverse spectral problems has become an important topic in the theory of differential equations.
See the monographs \cite{pt, yu} for detailed discussions and \cite{dt, eck-1, eck-2,e, g, gu, imt, it, mt, tru, wu} for some interesting results. In the study of inverse spectral theory, a key problem is to know whether or how spectral data can determine the potential. For the one dimensional case, it is well-known that two different sequences of eigenvalues, or one sequence of eigenvalues together with {norming constants} can recover the potential uniquely. See \cite{hm, marchen, ml,rs,yu} for more results. However, the existence of nontrivial families of isospectral potentials, discovered in both mathematical theory and physical models, show that spectral data alone may not always guarantee uniqueness. {From a physical viewpoint, isospectral operators correspond to different systems that exhibit identical energy spectra or vibration frequencies. See \cite{ert-1,ert-2,flmp,gk-1,gk-2,k-1,k-2,k-3,liucmp, w} for results on isospectrality. Among these results, a class of problems related to the separability of potentials has been of particular interest.}
For example,
let ${\rm Iso}(V)=\{\tilde{V}\in L^2(T):{\rm spec}(\tilde{V})={\rm spec}({V})\}$ denote the isospectral set of $V$, where spec$(V)$ is the set of eigenvalues of the continuous Schr\"{o}dinger operator $-\Delta_{\rm continuous}+V$ and $T=\R^d/(q_1\Z\times\cdots\times q_d\Z)$ with $d\geq 2$ is a rectangular torus. Eskin, Ralston and Trubowitz proved in \cite{ert-1, ert-2} that $\tilde{V}$ is also completely separable if $V$ is completely separable and ${\rm spec}(V)={\rm spec}(\tilde{V})$,
here $V$ is completely separable means that it can be written as the form $V(x_1,x_2,\cdots,x_d)=\sum_{i=1}^dV_i(x_i)$ with $V_i\in L^2(\R/q_i\Z)$.
Later  Gordon and Kappeler in \cite{gk-1} proved that if $V$ and $\tilde{V}$ are isospectral, completely separable potentials, then the one-dimensional potentials $V_i$ and $\tilde{V}_i$ are also isospectral, $1\leq i\leq d$, up to constants. 

In this paper, we consider the discrete periodic Schr\"odinger equation
\begin{align}
	\Delta u(n)+V(n)u(n)=\lambda u(n),\ n\in\mathbb{Z}^d,\label{mainequation}
\end{align}
with the so-called Floquet-Bloch boundary condition
\begin{align}
	u(n+q_j\mathbf{e}_j)=e^{2\pi \mathbf{i}k_j}u(n),\ j=1,2,\cdots,d,\quad n\in\mathbb{Z}^d,\label{flobloboundary}
\end{align}
{where $\{\mathbf{e}_j\}_{j=1}^d$ is the standard basis in $\mathbb{Z}^d$} and $\Delta=\Delta_{\rm discrete}$ is the discrete Laplacian
on $\ell^2(\Z^d)$ defined as
\[(\Delta u)(n)=\sum_{\|n'-n\|=1}u(n'),\]
with $n=(n_1,n_2,\cdots,n_d)\in\Z^d$, $n'=(n'_1,n'_2,\cdots,n'_d)\in\Z^d$,
$\|n'-n\|=\sum_{l=1}^{d}|n_l-n'_l|$ and
$V$ is $\Gamma$-periodic, $\Gamma=q_1\mathbb{Z}\oplus q_2\mathbb{Z}\oplus\cdots\oplus q_d\mathbb{Z}$ and $q_1,q_2,\cdots,q_d$ are pairwise coprime positive integers.
It follows from the standard Floquet
theory that problem \eqref{mainequation}-\eqref{flobloboundary} can be realized as the eigenvalues problem of a $Q\times Q$ matrix $D_V(k)$, where $Q=q_1q_2\cdots q_d$ and $k=(k_1,k_2,\cdots,k_d)$. To understand such spectral problems for periodic operators, it is known that both Bloch variety and Fermi variety play a significant role \cite{aim,bat-1, bat-2, flm-1,gkt, kt,kv-2, liu22, liujmp, lmt,sh}. The Bloch variety of $\Delta+V$ is defined as
\begin{align}
	B(V)=\{(k,\lambda)\in\C^{d+1}:\det(D_V(k)-\lambda I)=0\},\nonumber
\end{align}
and given any $\lambda\in\C$, Fermi variety of $\Delta+V$ is defined as $F_{\lambda}(V)=\{k\in\C^d:(k,\lambda)\in B(V)\}$. Eskin-Ralston-Trubowitz and Gordon-Kappeler  proposed the problem concerning the relation between (Floquet) isospectrality and separability of potentials in  \cite{ert-1, ert-2, gk-1}, which was further studied in \cite{flmp,k-2,k-3,liucmp,liujde,liucpam}. A function $V$ on $\mathbb{Z}^d$ is called $(d_1,d_2,\cdots,d_r)$-separable with $\sum_{j=1}^rd_j=d$ and $r\geq 2$, $d_j\geq 1$ if there exist functions $V_j$ on $\mathbb{Z}^{d_j}$, $j=1,2,\cdots,r$, such that for any $n=(n_1,n_2,\cdots,n_d)\in \mathbb{Z}^d$,
\begin{align}
	V(n)=V_1(n_1,n_2,\cdots,n_{d_1})+V_2(n_{d_1+1},\cdots,n_{d_1+d_2})+\cdots+V_r(n_{d_1+d_2+\cdots+d_{r-1}+1},\cdots,n_d).\nonumber
\end{align}
Two potentials $V$ and $Y$ are said to be Floquet isospectral if $B(V)=B(Y)$. In other words, $\det(D_V(k)-\lambda I)=\det(D_Y(k)-\lambda I)$ for any $k\in\R^d$ and any $\lambda\in\C$.  Liu in \cite{liucpam} proposed the notion of Fermi isospectrality which is weaker than the Floquet isospectrality. Two potentials $V$ and $Y$ are said to be Fermi isospectral if for some given $\lambda_0\in\C$, $\det(D_V(k)-\lambda_0I)= \det(D_Y(k)-\lambda_0I)$ for any $k\in \R^d$.  In \cite{liucpam}, Liu proved the following results, which are much related to our paper.
\begin{theorem}\cite{liucpam}\label{theoliucpam}
	Let $q_j$, $j=1,\cdots,d$ be pairwise coprime positive integers with $d\geq 3$.
	\begin{itemize}
		\item[(I)] Assume that real-valued $\Gamma$-periodic functions $V$ and $Y$ are Fermi isospectral, and $V$ is $(d_1,d_2,\cdots,d_r)$-separable, then $Y$ is $(d_1,d_2,\cdots,d_r)$-separable.
		\item[(II)] Assume that complex-valued $\Gamma$-periodic functions $V=\oplus_{j=1}^rV_j$ and $Y=\oplus_{j=1}^rY_j$ are Fermi isospectral and $(d_1,d_2,\cdots,d_r)$-separable, then each couples of lower dimensional functions $V_j$ and $Y_j$ are Floquet isospectral {\rm(}up to a constant{\rm)}.
	\end{itemize}
\end{theorem}
Theorem \ref{theoliucpam} tells us that Fermi isospectrality implies the coincidence of $(d_1,d_2,\cdots,d_r)$-separability. However, we point out  that such a $(d_1,d_2,\cdots,d_r)$-separability is also a relatively strong property because it requires that the function has a complete separability in the directions $(n_1,n_2,\cdots,n_{d_1})$, $(n_{d_1+1},\cdots,n_{d_1+d_2})$, $\cdots,(n_{d_1+d_2+\cdots+d_{r-1}+1},\cdots,n_d)$. In fact, if a function $V$ has some separability, it may only be separable in some (or even two) directions. For example, $V(n_1,n_2,\cdots,n_d)=V_1(n_1,n_3,\cdots,n_d)+V_2(n_2,n_3,\cdots,n_d)$ is separable in the first two directions, but does not possess a separability when considering all directions at the same time. We will introduce a notion of weak separability (see Definition \ref{defstsep} and Definition \ref{defstseptilde}), it turns out that such weak separability
can cover a wide class of separable functions, including the $(d_1,d_2,\cdots,d_r)$-separability.
On the other hand, from a geometric point of view, Fermi isospectrality requires less information than Floquet isospectrality because Floquet isospectrality requires that two $(d+1)$-dimensional Bloch varieties are the same, while Fermi isospectrality only requires that the projections of Bloch varieties on $d$-dimensional subspace are the same. In this paper, we introduce a notion of generalized partial Fermi isospectrality, which is even weaker than the Fermi isospectrality, and we will show that such new notion of isospectrality is completely suitable to study the weak separability. In some sense, it only requires that the projections of Bloch varieties $B(V)$ and $B(Y)$ on some lower dimensional subspace are the same after translation.
As an application, we can prove that two potentials have the same $(d_1,d_2,\cdots,d_r)$-separability by assuming that the projections of Fermi varieties or  Bloch varieties on some $3$-dimensional subspace are the same. Thus we can improve and generalize  the results in \cite{liucpam}, where it was required the coincidence of the whole Fermi varieties. Besides, we prove that each couples of components of the generalized Fermi isospectral potentials are Floquet isospectral in some sense.

The rest part of this paper is organized as follows. In Section 2, we introduce the new separability of functions and prove some properties of such separability. In Section 3, we introduce the notion of generalized partial Fermi isospectrality and prove our main results.
Moreover, we also prove a result about the Floquet isospectrality of the components of two generalized Fermi isospectral potentials, which allows us to explore more complicated separability of two Fermi (Floquet) isospectral potentials.

\section{Weak separability}
\setcounter{equation}{0}
In this section, we will introduce a notion of weak separability, which can be regarded as a generalization of the so-called $(d_1,d_2,\cdots,d_r)$-separability. Besides, we will establish a relation between the weak separability and the $(d_1,d_2,\cdots,d_r)$-separability. Let $\sum_{j=1}^rd_j=d$ with $r\geq 2,d_j\geq1,\ j=1,2,\cdots,r$. For convenience, let $d_0=0$. For any $x=(x_1,x_2,\cdots,x_d)\in\mathbb{C}^d$ and any $m=1,2,\cdots,r$, we denote $\tilde{x}_m$ by
$$\tilde{x}_m=\left(x_{1+\sum_{j=0}^{m-1}d_j},x_{2+\sum_{j=0}^{m-1}d_j},\cdots,x_{\sum_{j=0}^{m}d_j}\right)\in \mathbb{C}^{d_m}.$$

Let $W$ be a fundamental domain for $\Gamma=q_1\mathbb{Z}\oplus q_2\mathbb{Z}\oplus \cdots\oplus q_d\mathbb{Z}$, that is
\begin{align}
W=\{n=(n_1,n_2,\cdots,n_d)\in\mathbb{Z}^d:0\leq n_j\leq q_j-1,j=1,2,\cdots,d\}.\nonumber
\end{align}
Define the discrete Fourier transform $\hat{V}(l)$ on $W$ for $\Gamma$-periodic function $V$ by
\begin{align}
\hat{V}(l)=\frac{1}{Q}\sum_{n\in W}V(n)e^{-2\pi \mathbf{i}\left(\sum_{j=1}^d\frac{l_jn_j}{q_j}\right)},\nonumber
\end{align}
which can be extended to $\mathbb{Z}^d$ periodically by the way $\hat{V}(l)=\hat{V}(n)$
for any $l\equiv n\mod\Gamma$. Let
\begin{align}
[V]=\frac{1}{Q}\sum_{n\in W}V(n)\nonumber
\end{align}
be the average of $V$ over the periodicity cell. It is obvious that $[V]=\hat{V}(0,0,\cdots,0)$.

\begin{definition}\label{defstsep}
Let
$1\leq s<t\leq d$, we say that the function $V$ is $(s,t)$-separable if there exist two functions $V_s:\mathbb{Z}^{d-1}\to\mathbb{C}$ and $V_t:\mathbb{Z}^{d-1}\to\mathbb{C}$ such that for any $n=(n_1,n_2,\cdots,n_d)\in\mathbb{Z}^d$,
\begin{align}
V(n_1,n_2,\cdots,n_d)=V_s({n}_1,{n}_2,\cdots,{n}_{t-1},{n}_{t+1},\cdots,{n}_d)+V_t({n}_1,{n}_2,\cdots,{n}_{s-1},{n}_{s+1},\cdots,{n}_d).\nonumber
\end{align}
\end{definition}

{Let us remark that the above definition can cover a wide class of separable potentials. Indeed, as long as a function $V$ on $\Z^d$ has some separability, there must exist some $s,t$ and $V_s:\Z^{d-1}\to\C$ and $V_t:\Z^{d-1}\to\C$ such that {$V(n)=V_s({n}_1,{n}_2,\cdots,{n}_{t-1},{n}_{t+1},\cdots,{n}_d)+V_t({n}_1,{n}_2,\cdots,{n}_{s-1},{n}_{s+1},\cdots,{n}_d).$ With possible permutation}, we can rewrite it as $V(n)=V_1(n_1,n_3,\cdots,n_d)+V_2(n_2,n_3,\cdots,n_d)=V_1(n_1,\tilde{n}_3)+V_2(n_2,\tilde{n}_3)$, where $\tilde{n}_3=(n_3,n_4,\cdots,n_d)$.
	The following notion is a natural generalization.

\begin{definition}\label{defstseptilde}
	Let $V$ be a complex function on $\Z^d$. We say that $V$ is $(d_1,d_2,\cdots,d_{r-1})\oplus d_r$-separable if there exist $V_{j}:\mathbb{Z}^{d_j+d_r}\to\mathbb{C}$, $j=1,2,\cdots,r-1$, such that for any $n=(n_1,\cdots,n_d)=(\tilde{n}_1,\tilde{n}_2,\cdots,\tilde{n}_r)\in\mathbb{Z}^d$,
	\begin{align}
		V(n_1,n_2,\cdots,n_d)=V_1(\tilde{n}_1,\tilde{n}_r)+V_2(\tilde{n}_2,\tilde{n}_r)+\cdots+V_{r-1}(\tilde{n}_{r-1},\tilde{n}_r).\nonumber
\end{align}\end{definition}
}

It is obvious that if the function $V$ is $(d_1,d_2,\cdots,d_r)$-separable, then $V$ is $(s,t)$-separable for any $1\leq i<m\leq r$ and any $s\in \{1+\sum_{j=0}^{i-1}d_j,2+\sum_{j=0}^{i-1}d_j,\cdots,\sum_{j=0}^{i}d_j\}$, $t\in \{1+\sum_{j=0}^{m-1}d_j,2+\sum_{j=0}^{m-1}d_j,\cdots,\sum_{j=0}^{m}d_j\}$.
{Besides, if the function $V$ is $(d_1,d_2,\cdots,d_{r-1})\oplus d_r$-separable, then $V$ is $(s,t)$-separable for any $1\leq i<m\leq r-1$ and any $s\in \{1+\sum_{j=0}^{i-1}d_j,2+\sum_{j=0}^{i-1}d_j,\cdots,\sum_{j=0}^{i}d_j\}$, $t\in \{1+\sum_{j=0}^{m-1}d_j,2+\sum_{j=0}^{m-1}d_j,\cdots,\sum_{j=0}^{m}d_j\}$.
We will show that the reverse results also hold.}

\begin{lemma}\label{stseparablefermi}
Let $1\leq s<t\leq d$. A $\Gamma$-periodic function $V:\Z^d\to \C$ is $(s,t)$-separable if and only if $\hat{V}(l)=0$ for any $l=(l_1,l_2,\cdots,l_d)\in W$ with $l_s\neq 0,l_t\neq 0$.
\end{lemma}
\begin{proof}
Suppose that $V$ is $(s,t)$-separable, then there exist $V_s:\Z^{d-1}\to\C$ and $V_t:\Z^{d-1}\to\C$ such that for any $n=(n_1,n_2,\cdots,n_d)\in\Z^d$,
\begin{align}
V(n)=V_s({n}_1,{n}_2,\cdots,{n}_{t-1},{n}_{t+1},\cdots,{n}_d)+V_t({n}_1,{n}_2,\cdots,{n}_{s-1},{n}_{s+1},\cdots,{n}_d).\nonumber
\end{align}
Then for any $l=(l_1,l_2,\cdots,l_d)\in W$ with $l_s\neq 0$, $l_t\neq 0$, one has
\begin{align}
\hat{V}(l)=&\frac{1}{Q}\sum_{n\in W}V(n)e^{-2\pi \mathbf{i}\sum_{j=1}^d\frac{n_jl_j}{q_j}}\nonumber\\
=&\frac{1}{Q}\sum_{n\in W}V_s({n}_1,{n}_2,\cdots,{n}_{t-1},{n}_{t+1},\cdots,{n}_d)e^{-2\pi \mathbf{i}\sum_{j=1}^d\frac{n_jl_j}{q_j}}\nonumber\\
&+\frac{1}{Q}\sum_{n\in W}V_t({n}_1,{n}_2,\cdots,{n}_{s-1},{n}_{s+1},\cdots,{n}_d)e^{-2\pi \mathbf{i}\sum_{j=1}^d\frac{n_jl_j}{q_j}}.\nonumber
\end{align}
Since for any $l_s\in \{1,2,\cdots,q_s-1\}$ and $l_t\in \{1,2,\cdots,q_t-1\}$, one has
$$\sum_{n_s=0}^{q_s-1}e^{-2\pi\mathbf{i}\frac{n_sl_s}{q_s}}=0,\  \sum_{n_t=0}^{q_t-1}e^{-2\pi\mathbf{i}\frac{n_tl_t}{q_t}}=0.$$
Thus $\hat{V}(l)=0$.

On the other hand, assume that $\hat{V}(l)=0$ for any $l\in W$ with $l_s\neq0,l_t\neq 0$. Then by the inverse discrete Fourier transform we have for any $n=(n_1,n_2,\cdots,n_d)\in \Z^d$,
\begin{align}
V(n)=&\sum_{l\in W}\hat{V}(l)e^{2\pi\mathbf{i}\sum_{j=1}^d\frac{n_jl_j}{q_j}}\nonumber\\
=&\sum_{l\in W,l_s=0,l_t\neq 0}\hat{V}(l)e^{2\pi\mathbf{i}\sum_{j=1}^d\frac{n_jl_j}{q_j}}+\sum_{l\in W,l_t=0}\hat{V}(l)e^{2\pi\mathbf{i}\sum_{j=1}^d\frac{n_jl_j}{q_j}}.\nonumber
\end{align}
Taking
\begin{align}
V_t(n_1,n_2,\cdots,n_{s-1},n_{s+1},\cdots,n_d)&=\sum_{l\in W,l_s=0,l_t\neq 0}\hat{V}(l)e^{2\pi\mathbf{i}\sum_{j=1}^d\frac{n_jl_j}{q_j}},\nonumber\\
V_s(n_1,n_2,\cdots,n_{t-1},n_{t+1},\cdots,n_d)&=\sum_{l\in W,l_t=0}\hat{V}(l)e^{2\pi\mathbf{i}\sum_{j=1}^d\frac{n_jl_j}{q_j}}.\nonumber
\end{align}
Then $V(n)=V_s(n_1,n_2,\cdots,n_{t-1},n_{t+1},\cdots,n_d)+V_t(n_1,n_2,\cdots,n_{s-1},n_{s+1},\cdots,n_d)$ and thus $V$ is $(s,t)$-separable.
\end{proof}

In the following we use $\mathbf{0}$ to denote the zero vector in the corresponding space, whose dimension may change {even in the same formula}.

{\begin{theorem}\label{theoremoplussp}
	Let $V$ be a complex $\Gamma$-periodic function on $\mathbb{Z}^d$. Then $V$ is $(d_1,d_2,\cdots,d_{r-1})\oplus d_r$-separable if and only if
	$V$ is $(s,t)$-separable for any $1\leq i<m\leq r-1$ and for any $s\in \{1+\sum_{j=0}^{i-1}d_j,2+\sum_{j=0}^{i-1}d_j,\cdots,\sum_{j=0}^{i}d_j\}$, $t\in \{1+\sum_{j=0}^{m-1}d_j,2+\sum_{j=0}^{m-1}d_j,\cdots,\sum_{j=0}^{m}d_j\}$.
\end{theorem}
\begin{proof}
	We only need to prove the necessity. Assume that $V$ is $(s,t)$-separable for any $1\leq i<m\leq r-1$ and for any $s\in \{1+\sum_{j=0}^{i-1}d_j,2+\sum_{j=0}^{i-1}d_j,\cdots,\sum_{j=0}^{i}d_j\}$, $t\in \{1+\sum_{j=0}^{m-1}d_j,2+\sum_{j=0}^{m-1}d_j,\cdots,\sum_{j=0}^{m}d_j\}$.
	 It follows from Lemma \ref{stseparablefermi} that $\hat{V}(l)=0$ for any $l=(l_1,l_2,\cdots,l_d)\in W$ with $l_s\neq 0,l_t\neq 0$, where $s\in \{1+\sum_{j=0}^{i-1}d_j,2+\sum_{j=0}^{i-1}d_j,\cdots,\sum_{j=0}^{i}d_j\}$, $t\in \{1+\sum_{j=0}^{m-1}d_j,2+\sum_{j=0}^{m-1}d_j,\cdots,\sum_{j=0}^{m}d_j\}$ and $1\leq i<m\leq r-1$. Therefore, we have that $\hat{V}(l)=0$ for any $l=(l_1,l_2,\cdots,l_d)=(\tilde{l}_1,\tilde{l}_2,\cdots,\tilde{l}_r)$ with at least two non-zero $\tilde{l}_j$, $j=1,2,\cdots,r-1$. Then
\begin{align}
	V(n)=&\sum_{l\in W}\hat{V}(l)e^{2\pi\mathbf{i}\sum_{j=1}^d\frac{n_jl_j}{q_j}}\nonumber\\
	=&\left(\sum_{\substack{l\in W\\ \tilde{l}_2=\mathbf{0},\tilde{l}_3=\mathbf{0},\cdots,\tilde{l}_{r-1}=\mathbf{0}}}+\sum_{\substack{l\in W\\ \tilde{l}_1=\mathbf{0},\tilde{l}_3=\mathbf{0},\cdots,\tilde{l}_{r-1}=\mathbf{0}}}+\cdots+\sum_{\substack{l\in W\\ \tilde{l}_1=\mathbf{0},\tilde{l}_2=\mathbf{0},\cdots,\tilde{l}_{r-2}=\mathbf{0}}}\right)\hat{V}(l)e^{2\pi\mathbf{i}\sum_{j=1}^d\frac{n_jl_j}{q_j}}\nonumber\\
	&-(r-2)\times\sum_{\substack{l\in W\\ \tilde{l}_1=\mathbf{0},\tilde{l}_2=\mathbf{0},\cdots,\tilde{l}_{r-1}=\mathbf{0}}}\hat{V}(l)e^{2\pi\mathbf{i}\sum_{j=1}^d\frac{n_jl_j}{q_j}}.\nonumber
\end{align}
By taking
\begin{align}
	V_1(\tilde{n}_1,\tilde{n}_r)=\sum_{\substack{l\in W\\ \tilde{l}_2=\mathbf{0},\tilde{l}_3=\mathbf{0},\cdots,\tilde{l}_{r-1}=\mathbf{0}}}\hat{V}(l)e^{2\pi\mathbf{i}\sum_{j=1}^d\frac{n_jl_j}{q_j}}-(r-2)\times\sum_{\substack{l\in W\\ \tilde{l}_1=\mathbf{0},\tilde{l}_2=\mathbf{0},\cdots,\tilde{l}_{r-1}=\mathbf{0}}}\hat{V}(l)e^{2\pi\mathbf{i}\sum_{j=1}^d\frac{n_jl_j}{q_j}},\nonumber
\end{align}
and
\begin{align}
	V_j(\tilde{n}_j,\tilde{n}_r)&=\sum_{\substack{l\in W\\ \tilde{l}_1=\mathbf{0},\cdots,\tilde{l}_{j-1}=\mathbf{0},\\\tilde{l}_{j+1}=\mathbf{0},\cdots,\tilde{l}_{r-1}=\mathbf{0}}}\hat{V}(l)e^{2\pi\mathbf{i}\sum_{j=1}^d\frac{n_jl_j}{q_j}}, \ j=2,3,\cdots,r-2,\nonumber\\
	V_{r-1}(\tilde{n}_{r-1},\tilde{n}_r)&=\sum_{\substack{l\in W\\ \tilde{l}_1=\mathbf{0},\tilde{l}_2=\mathbf{0},\cdots,\tilde{l}_{r-2}=\mathbf{0}}}\hat{V}(l)e^{2\pi\mathbf{i}\sum_{j=1}^d\frac{n_jl_j}{q_j}},\nonumber
\end{align}
we can see that $V$ is $(d_1,d_2,\cdots,d_{r-1})\oplus d_r$-separable.
\end{proof}
}

As a consequence of Theorem \ref{theoremoplussp}, the following result holds.
\begin{theorem}\label{d1d2drnewsepafermith}
	Let $V$ be a complex $\Gamma$-periodic function on $\mathbb{Z}^d$. Then $V$ is $(d_1,d_2,\cdots,d_r)$-separable if and only if
	$V$ is $(s,t)$-separable for any $1\leq i<m\leq r$ and for any $s\in \{1+\sum_{j=0}^{i-1}d_j,2+\sum_{j=0}^{i-1}d_j,\cdots,\sum_{j=0}^{i}d_j\}$, $t\in \{1+\sum_{j=0}^{m-1}d_j,2+\sum_{j=0}^{m-1}d_j,\cdots,\sum_{j=0}^{m}d_j\}$.
\end{theorem}
\begin{proof}
Although we have assumed that each $d_j\geq 1$ in the above definitions, we can let $d_r=0$ in the proof of Theorem \ref{theoremoplussp}. Then $(d_1,d_2,\cdots,d_{r-1})\oplus d_r$-separability becomes $(d_1,d_2,\cdots,d_{r-1})$-separability. Now the result follows from Theorem \ref{theoremoplussp}.
\end{proof}

\section{Main results}
\setcounter{equation}{0}
Recall that problem \eqref{mainequation}-\eqref{flobloboundary} can be realized by the eigenvalue problem for a $Q\times Q$ matrix $D_V(k)$.
For example, when $d=1$, $D_V(k)$ is the following $q_1\times q_1$ matrix
\begin{align}
D_V(k_1)=\begin{pmatrix}
V(1)& 1& 0&\cdots&0&\cdots& 0& e^{-2\pi \mathbf{i}k_1}\\
1& V(2)& 1& 0&\cdots& 0&\cdots&0\\
0&1& V(3)& 1& 0&\cdots&0&0\\
\vdots&&&\ddots&&&&\vdots\\
\vdots&&&&\ddots&&&\vdots\\
0&&&&&\ddots&&0\\
0&\cdots&0&\cdots&0&1&V(q_1-1)&1\\
e^{2\pi \mathbf{i}k_1}&0&\cdots&0&\cdots&0&1&V(q_1)
\end{pmatrix}.\nonumber
\end{align}
Let $z=(z_1,z_2,\cdots,z_d)\in \mathbb{C}^d$ with $z_j=e^{2\pi \mathbf{i}k_j},j=1,2,\cdots,d,$ and $\mathcal{D}_V(z)=D_V(k)$,
\begin{align}
\mathcal{P}_V(z,\lambda)=P_V(k,\lambda)=\det(D_V(k)-\lambda I).\label{pvzpvkfer}
\end{align}
It follows that $\mathcal{P}_V(z,\lambda)$ is a polynomial in $\lambda$ and $z_1,z_1^{-1},\cdots,z_d,z_d^{-1}$.
Namely, $\mathcal{P}_V(z,\lambda)$ is a Laurent polynomial of $\lambda,z_1,\cdots,z_d$.
Let $\tilde{\mathcal{D}}_V(z)=\tilde{\mathcal{D}}_V(z_1,z_2,\cdots,z_d)=\mathcal{D}_V(z_1^{q_1},z_2^{q_2},\cdots,z_d^{q_d})$
and
\begin{align}
\tilde{\mathcal{P}}_V(z,\lambda)=\det(\tilde{\mathcal{D}}_V(z)-\lambda I)=\mathcal{P}_V(z_1^{q_1},z_2^{q_2},\cdots,z_d^{q_d},\lambda).\label{pvztilde}
\end{align}
For $0\leq n_j\leq q_j-1$ and $j=1,2,\cdots,d$, let
\[\rho^j_{n_j}=e^{2\pi \mathbf{i}\frac{n_j}{q_j}}.\]
The following result is a direct application of the discrete Floquet transform. See for example \cite{flm-2,ku}.

\begin{lemma} \label{lemunitary}Let $n=(n_1,n_2,\cdots,n_d)\in W$ and $n'=(n_1',n_2',\cdots,n_d')\in W$. Then $\tilde{\mathcal{D}}_V(z)$ is unitarily equivalent to $A_z+B_V$, where $A_z$ is a diagonal matrix with entries
\[A_z(n;n')=\left(\sum_{j=1}^d\left(\rho^j_{n_j}z_j+\frac{1}{\rho^j_{n_j}z_j}\right)\right)\delta_{(n;n')}\quad {\rm with}\quad \delta_{(n;n')}=\left\{\begin{array}{l}
	1, \quad\  n=n',\\
	0,~\quad n\neq n',
\end{array}\right.\]
and
\[B_V(n;n')=\hat{V}(n_1-n_1',n_2-n_2',\cdots,n_d-n_d').\]
In particular, $\tilde{\mathcal{P}}_V(z,\lambda)=\det(A_z+B_V-\lambda I).$
\end{lemma}
{Two potentials $V$ and $Y$ are said to be Fermi isospectral if for some given $\lambda_0\in\C$, $P_V(k,\lambda_0)=P_Y(k,\lambda_0)$ for any $k\in\R^d$. Such a notion of Fermi isospectrality was proposed by Liu (\cite[Remark 2.6]{liucpam}). {Now we introduce the new notions of isospectrality, which are weaker than the Fermi isospectrality.}
\begin{definition}\label{defferisotrans}
Let $d\geq 2$. We say that two functions $V$ and $Y$ are generalized partially Fermi isospectral if there exist $\lambda_1,\lambda_2\in\C$, non-empty set $S\subset\{1,2,\cdots,d\}$, and $k_j^*\in\R$ for each $j\in \{1,2,\cdots,d\}\setminus S$, such that for any $k_j\in \R$, $j\in S$,
\begin{align}
P_V(k,\lambda_1)=P_Y(k,\lambda_2),\label{pvlambda1pylambda2iden2}
\end{align}
where $k=(k_1,k_2,\cdots,k_d)$ with $k_j=k_j^*$ if $j\in \{1,2,\cdots,d\}\setminus S$. In particular, if $\lambda_1=\lambda_2$, we say that $V$ and $Y$ are partially Fermi isospectral. If $S=\{1,2,\cdots,d\}$ and {there exist $\lambda_1,\lambda_2\in\C$ such that \eqref{pvlambda1pylambda2iden2} holds}, we say that $V$ and $Y$ are generalized Fermi isospectral.
\end{definition}
\begin{remark}
	It is obvious that the above definitions depend on the set $S$ and $\lambda_1,\lambda_2$, but we prefer to leave the dependence on $S,\lambda_1$ and $\lambda_2$ implicit if there is no confusion. For the special case $S=\{1,2,\cdots,d\}$ and $\lambda_1=\lambda_2$, the generalized partial Fermi isospectrality become the notion of Fermi isospectrality. Moreover, it is clear that for any non-empty set $S'\subset S$, if $V$ and $Y$ are {generalized} partially Fermi isospectral depending on $S,\lambda_1$ and $\lambda_2$, then $V$ and $Y$ are also {generalized} partially Fermi isospectral depending on $S',\lambda_1$ and $\lambda_2$.
\end{remark}
}

\begin{theorem}\label{v-yequlambda1-2}
Let $V$ and $Y$ be complex $\Gamma$-periodic functions on $\Z^d$. If there exist $\lambda_1,\lambda_2\in\C$ and the set $S\subset\{1,2,\cdots,d\}$ with the cardinality $\# S\geq 2$, such that $V$ and $Y$ are generalized partially Fermi isospectral,
then $[V]-[Y]=\lambda_1-\lambda_2$. {In particular, $[V]=[Y]$ if $V$ and $Y$ are  partially Fermi isospectral.}
\end{theorem}
\begin{proof} Without loss of generality, assume $S=\{1,2\}$.
By \eqref{pvzpvkfer}, \eqref{pvztilde} and \eqref{pvlambda1pylambda2iden2}, there exist $k_j^*\in\R$, $j=3,\cdots,d$, such that for any $(k_1,k_2)\in\R^2$,
$\tilde{\mathcal{P}}_V(z,\lambda_1)\equiv\tilde{\mathcal{P}}_Y(z,\lambda_2)$, where
\begin{align}
z=(z_1,z_2,z_3^*\cdots,z_d^*)=\left(e^{\frac{2\pi\mathbf{i} k_1}{q_1}},e^{\frac{2\pi\mathbf{i} k_{2}}{q_{2}}},e^{\frac{2\pi\mathbf{i} k_{3}^*}{q_{3}}},\cdots,e^{\frac{2\pi\mathbf{i} k_d^*}{q_d}}\right).\nonumber
\end{align}
Expanding $\tilde{\mathcal{P}}_V(z,\lambda_1)$ as a Laurent polynomial of $z_1,z_2$ and by Lemma \ref{lemunitary}, we obtain that
{\small\begin{align}
\tilde{\mathcal{P}}_V(z,\lambda_1)\equiv \prod_{n\in W}&\left(\rho^1_{n_1}z_1+\frac{1}{\rho^1_{n_1}z_1}+\rho^2_{n_2}z_2+\frac{1}{\rho^2_{n_2}z_2}+\sum_{i=3}^d\left(\rho^i_{n_i}z_i^*+\frac{1}{\rho^i_{n_i}z_i^*}\right)+[V]-\lambda_1\right)+(\cdots),\nonumber
\end{align}}where $(\cdots)$ consists of the terms of $z_1^{a_1}z_2^{a_2}$ with $a_1+a_2\leq Q-2$. Then we know that the terms of $z_1^{a_1}z_2^{a_2}$ with $a_1+a_2=Q-1$ equal
\begin{align}
\sum_{n\in W}\frac{\prod_{l\in W}\left(\rho^1_{l_1}z_1+ \rho^2_{l_2}z_2\right)}{\rho^1_{n_1}z_1+\rho^2_{n_2}z_2}\left(\sum_{i=3}^d\left(\rho^i_{n_i}z_i^*+\frac{1}{\rho^i_{n_i}z_i^*}\right)+[V]-\lambda_1\right).\nonumber
\end{align}
Thus by $\tilde{\mathcal{P}}_V(z,\lambda_1)\equiv\tilde{\mathcal{P}}_Y(z,\lambda_2)$ we have that for any possible $z_1\in\C,z_2\in\C$,
\begin{align}
\sum_{\substack{0\leq n_1\leq q_1-1\\0\leq n_2\leq q_2-1}}\frac{[V]-\lambda_1}{\rho^1_{n_1}z_1+\rho^2_{n_2}z_2}=\sum_{\substack{0\leq n_1\leq q_1-1\\0\leq n_2\leq q_2-1}}\frac{[Y]-\lambda_2}{\rho^1_{n_1}z_1+\rho^2_{n_2}z_2}.\label{midfervl1yl2}
\end{align}
Since $q_1$ and $q_2$ are coprime,  we have $\rho^1_{n_1}-\rho^2_{n_2}\neq 0$ for any $(n_1,n_2)\neq (0,0)$ (module periodicity). Let $z_1=1$ and $z_2\to-1$, by \eqref{midfervl1yl2} we have
\begin{align}
\frac{[V]-\lambda_1}{1+z_2}+O(1)=\frac{[Y]-\lambda_2}{1+z_2}+O(1),\nonumber
\end{align}
which implies that $[V]-[Y]=\lambda_1-\lambda_2$.
\end{proof}

\begin{lemma}Let $V$ and $Y$ be real $\Gamma$-periodic functions on $\Z^d$. If there exist $\lambda_1,\lambda_2\in\C$ and the set $S\subset\{1,2,\cdots,d\}$ with the cardinality $\# S\geq 2$, such that $V$ and $Y$ {are generalized partially Fermi isospectral},
then for all possible $z_j\in\C$, $j\in S$,
\begin{align}\label{sumv2y2}
\sum_{n\in W}\sum_{l\in W}\frac{\abs{\hat{V_1}(l)}^2}{\left(\sum_{j\in S}\rho^j_{n_j}z_j\right)\left(\sum_{j\in S}\rho^j_{n_j+l_j}z_j\right)}=\sum_{n\in W}\sum_{l\in W}\frac{\abs{\hat{Y_1}(l)}^2}{\left(\sum_{j\in S}\rho^j_{n_j}z_j\right)\left(\sum_{j\in S}\rho^j_{n_j+l_j}z_j\right)},
\end{align}
where $n=(n_1,n_2,\cdots,n_d)$ and $l=(l_1,l_2,\cdots,l_d)$, $V_1(n)=V(n)-\lambda_1$, $Y_1(n)=Y(n)-\lambda_2$.
\end{lemma}
\begin{proof}Without loss of generality, we assume that $S=\{1,2,\cdots,d'\}$ with $2\leq d'\leq d$.
One may readily verify that
\begin{align}
\hat{V}_1(l)=
\hat{V}{(l)}-\lambda_1\delta_{(l;\mathbf{0})},\   \hat{Y}_1(l)=
\hat{Y}{(l)}-\lambda_2\delta_{(l;\mathbf{0})},\label{v1vy1y}
\end{align}
where $\delta_{(l;\mathbf{0})}=1$ if $l=\mathbf{0}$ and $\delta_{(l;\mathbf{0})}=0$ if $l\neq\mathbf{0}$.
By \eqref{pvzpvkfer}, \eqref{pvztilde} and \eqref{pvlambda1pylambda2iden2}, there exist $k_j^*\in\R$, $j=d'+1,\cdots,d$, such that for any $(k_1,k_2,\cdots,k_{d'})\in\R^{d'}$,
$\tilde{\mathcal{P}}_V(z,\lambda_1)\equiv\tilde{\mathcal{P}}_Y(z,\lambda_2)$, where
\begin{align}
z=(z_1,\cdots,z_{d'},z_{d'+1}^*,\cdots,z_d^*)=\left(e^{\frac{2\pi\mathbf{i} k_1}{q_1}},\cdots,e^{\frac{2\pi\mathbf{i} k_{d'}}{q_{d'}}},e^{\frac{2\pi\mathbf{i} k_{d'+1}^*}{q_{d'+1}}},\cdots,e^{\frac{2\pi\mathbf{i} k_d^*}{q_d}}\right).\nonumber
\end{align}
Let $\hat{z}=(z_1,z_2,\cdots,z_{d'})$ and $h(\hat{z})=\prod_{n\in W}\left(\sum_{j=1}^{d'}\rho^j_{n_j}z_j\right)$. Since $V$ and $Y$ are real, by Lemma \ref{lemunitary} we know that $B_V$ and $B_Y$ are Hermitian.
Expanding $\tilde{\mathcal{P}}_V(z,\lambda_1)$ as a Laurent polynomial of $z_1,z_2,\cdots,z_{d'}$, by Lemma \ref{lemunitary} we obtain that
\begin{align}
\tilde{\mathcal{P}}_V(z,\lambda_1)
=&\prod_{n\in W}\left([V]-\lambda_1+\sum_{j=1}^{d'}\left(\rho^j_{n_j}z_j+\frac{1}{\rho^j_{n_j}z_j}\right)+\sum_{j=d'+1}^{d}\left(\rho^j_{n_j}z_j^*+\frac{1}{\rho^j_{n_j}z_j^*}\right)\right)\nonumber\\
&-\frac{1}{2}\sum_{\substack{n\in W,n'\in W\\n\neq n'}}\frac{h(\hat{z})}{\left(\sum_{j=1}^{d'}\rho^j_{n_j}z_j\right)\left(\sum_{j=1}^{d'}\rho^j_{n'_j}z_j\right)}\left|{\hat{V}(n-n')}\right|^2+(\cdots),\nonumber
\end{align}
where $n'=(n_1',n_2',\cdots,n_d')$ and $(\cdots)$ consists of terms of $z_1^{a_1}z_2^{a_2}\cdots z_{d'}^{a_{d'}}$ with $\sum_{j=1}^{d'}a_j\leq Q-3$. It follows from Theorem \ref{v-yequlambda1-2} that $[V]-\lambda_1=[Y]-\lambda_2$.
Hence,
\begin{align}
\tilde{\mathcal{P}}_V&(z,\lambda_1)-\prod_{n\in W}\left([V]-\lambda_1+\sum_{j=1}^{d'}\left(\rho^j_{n_j}z_j+\frac{1}{\rho^j_{n_j}z_j}\right)+\sum_{j=d'+1}^{d}\left(\rho^j_{n_j}z_j^*+\frac{1}{\rho^j_{n_j}z_j^*}\right)\right)\nonumber\\
\equiv &\tilde{\mathcal{P}}_Y(z,\lambda_2)-\prod_{n\in W}\left([Y]-\lambda_2+\sum_{j=1}^{d'}\left(\rho^j_{n_j}z_j+\frac{1}{\rho^j_{n_j}z_j}\right)+\sum_{j=d'+1}^{d}\left(\rho^j_{n_j}z_j^*+\frac{1}{\rho^j_{n_j}z_j^*}\right)\right).\nonumber
\end{align}
Comparing the highest terms of the equality, namely $z_1^{a_1}z_2^{a_2}\cdots z_{d'}^{a_{d'}}$ with $\sum_{j=1}^{d'}a_j= Q-2$, one can obtain that for any $\hat{z}\in\mathbb{C}^{d'}$,
\begin{align}
\sum_{\substack{n \in W, n' \in W \\ n \neq n'}} \frac{h(\hat{z}) \left| \hat{V}(n - n') \right|^2}{\left(\sum_{j=1}^{d'}\rho^j_{n_j}z_j\right)\left(\sum_{j=1}^{d'}\rho^j_{n'_j}z_j\right)} = \sum_{\substack{n \in W, n' \in W \\ n \neq n'}} \frac{h(\hat{z}) \left| \hat{Y}(n - n') \right|^2}{\left(\sum_{j=1}^{d'}\rho^j_{n_j}z_j\right)\left(\sum_{j=1}^{d'}\rho^j_{n'_j}z_j\right)}.\nonumber
\end{align}
Due to \eqref{v1vy1y} and the fact $[V]-\lambda_1=[Y]-\lambda_2$, we have for all possible $\hat{z}\in \C^{d'}$,
\begin{align}
\sum_{\substack{n \in W, n' \in W}} \frac{\left| \hat{V_1}(n - n') \right|^2}{\left(\sum_{j=1}^{d'}\rho^j_{n_j}z_j\right)\left(\sum_{j=1}^{d'}\rho^j_{n'_j}z_j\right)} = \sum_{\substack{n \in W, n' \in W}} \frac{\left| \hat{Y_1}(n - n') \right|^2}{\left(\sum_{j=1}^{d'}\rho^j_{n_j}z_j\right)\left(\sum_{j=1}^{d'}\rho^j_{n'_j}z_j\right)}.\nonumber
\end{align}
Changing variables with $l=n-n'$, we conclude that
\begin{align}
\sum_{n\in W}&\sum_{l\in\{-n\}+W}\frac{\left|{\hat{V_1}(l)}\right|^2}{\left(\sum_{j=1}^{d'}\rho^j_{n_j}z_j\right)\left(\sum_{j=1}^{d'}\rho^j_{n_j+l_j}z_j\right)}\nonumber\\
&=\sum_{n\in W}\sum_{l\in\{-n\}+W}\frac{\left|{\hat{Y_1}(l)}\right|^2}{\left(\sum_{j=1}^{d'}\rho^j_{n_j}z_j\right)\left(\sum_{j=1}^{d'}\rho^j_{n_j+l_j}z_j\right)},\nonumber
\end{align}
where $\{n\}+W$ denotes the set $\{n+n'\}_{n'\in W}$. For each fixed $n\in W$, since
\begin{align}
\frac{\left|{\hat{V_1}(l)}\right|^2}{\left(\sum_{j=1}^{d'}\rho^j_{n_j}z_j\right)\left(\sum_{j=1}^{d'}\rho^j_{n_j+l_j}z_j\right)}\ \ \mathrm{and}\ \ \frac{\left|{\hat{Y_1}(l)}\right|^2}{\left(\sum_{j=1}^{d'}\rho^j_{n_j}z_j\right)\left(\sum_{j=1}^{d'}\rho^j_{n_j+l_j}z_j\right)}\nonumber
\end{align}
are $\Gamma$-periodic in $l$,  we can obtain \eqref{sumv2y2}.
\end{proof}

\begin{lemma}\cite{liucpam}\label{lemcoprimeliucpam}
	Assume that $q_1,q_2,q_3$ are pairwise coprime and
	\begin{align}
		\begin{vmatrix}
			1 & 1 & 1 \\
			\\
			\rho^1_{l_1} & \rho^2_{l_2} &\rho^3_{l_3}\\
			\\
			\rho^1_{l_1'} & \rho^2_{l_2'} &\rho^3_{l_3'}
		\end{vmatrix}= 0\quad {\rm with}\quad l_j,l_j'\in\{0,1,\cdots,q_j-1\},j=1,2,3.\nonumber
	\end{align}
	Then $(l_1,l_2,l_3)$ and $(l_1',l_2',l_3')$ must fall into one of the following cases:
	\[(a)~l_1=l_2=l_3=0;~(b)~l_1'=l_2'=l_3'=0;~(c)~l_1=l_1',l_2=l_2',l_3=l_3';\]\[(d)~l_1=l_1'=0,l_2=l_2'=0;~(e)~l_1=l_1'=0,l_3=l_3'=0;~(f)~l_2=l_2'=0,l_3=l_3'=0.\]
\end{lemma}

{\begin{theorem}\label{thmferst}
Let $V$ and $Y$ be real $\Gamma$-periodic functions on $\Z^d$.  If there exist $\lambda_1,\lambda_2\in\C$  and the set $S\subset\{1,2,\cdots,d\}$ with the cardinality $\# S\geq 3$, such that $V$ and $Y$ are generalized partially Fermi isospectral,
then for any $s,t\in S$ with $s<t$, if $V$ is $(s,t)$-separable, $Y$ is also $(s,t)$-separable.
\end{theorem}}
\begin{proof}
Without loss of generality, we assume that $S=\{1,2,3\}$ and $(s,t)=(1,2)$.
By \eqref{sumv2y2} one has for any possible $z_j$, $j=1,2,3$,
\begin{align}
\sum_{\substack{0\leq n_i\leq q_i-1\\ i=1,2,3}}&\sum_{l'\in W}\frac{\abs{\hat{V_1}(l')}^2}{\left(\sum_{j=1}^3\rho^j_{n_j}z_j\right)\left(\sum_{j=1}^3\rho^j_{n_j+l_j'}z_j\right)}\nonumber\\
=&\sum_{\substack{0\leq n_i\leq q_i-1\\ i=1,2,3}}\sum_{l'\in W}\frac{\abs{\hat{Y_1}(l')}^2}{\left(\sum_{j=1}^3\rho^j_{n_j}z_j\right)\left(\sum_{j=1}^3\rho^j_{n_j+l_j'}z_j\right)},\label{midj123}
\end{align}
where $l'=(l_1',l_2',\cdots,l_d')$, $V_1(n)=V(n)-\lambda_1$ and $Y_1(n)=Y(n)-\lambda_2$.

Since $q_j$, $j=1,2,3$ are pairwise coprime, we have that for any $(n_1,n_2,n_3)\neq \mathbf{0}$ (module periodicity), the two planes
\begin{align}
\sum_{j=1}^{{3}}z_j=0 \ \ \mathrm{and}\ \  \sum_{j=1}^{{3}}\rho^j_{n_j}z_j=0\nonumber
\end{align}
are not parallel. Then there exists  $\hat{z}^*=(z_1^*,z_2^*,z_{3}^*)$ such that $z_1^*+z_2^*+z_3^*=0$
and for any $(n_1,n_2,n_3)\neq \mathbf{0}$ (module periodicity),
$
\sum_{j=1}^{{3}}\rho^j_{n_j}z_j^*\neq 0.
$
Let $z_2=z_2^*$, $z_3=z_3^*$, and $z_1\to z_1^*$, by \eqref{midj123} one obtains
\begin{align}
\sum_{\substack{l'\in W\\ l_1'=l_2'=l_3'=0}}\frac{\abs{\hat{V}_1(l')}^2}{(z_1-z_1^*)^2}+\frac{O(1)}{\abs{z_1-z_1^*}}=\sum_{\substack{l'\in W\\ l_1'=l_2'=l_3'=0}}\frac{\abs{\hat{Y}_1(l')}^2}{(z_1-z_1^*)^2}+\frac{O(1)}{\abs{z_1-z_1^*}},\nonumber
\end{align}
which implies that
\begin{align}
\sum_{\substack{l'\in W\\ l_1'=l_2'=l_3'=0}}\abs{\hat{V}_1(l')}^2=\sum_{\substack{l'\in W\\ l_1'=l_2'=l_3'=0}}\abs{\hat{Y}_1(l')}^2.\nonumber
\end{align}
Thus for any fixed $n_1,n_2,n_3$, and possible $z_j,j=1,2,3,$
\begin{align}
\sum_{\substack{l'\in W\\l_1'=l_2'=l_3'=0}}\frac{\abs{\hat{V_1}(l')}^2}{\left(\sum_{j=1}^3\rho^j_{n_j}z_j\right)^2}=\sum_{\substack{l'\in W\\l_1'=l_2'=l_3'=0}}\frac{\abs{\hat{Y_1}(l')}^2}{\left(\sum_{j=1}^3\rho^j_{n_j}z_j\right)^2},\nonumber
\end{align}
then we have
\begin{align}
\sum_{\substack{0\leq n_i\leq q_i-1\\i=1,2,3}}\sum_{\substack{l'\in W\\l_1'=l_2'=l_3'=0}}\frac{\abs{\hat{V_1}(l')}^2}{\left(\sum_{j=1}^3\rho^j_{n_j}z_j\right)^2}=\sum_{\substack{0\leq n_i\leq q_i-1\\i=1,2,3}}\sum_{\substack{l'\in W\\l_1'=l_2'=l_3'=0}}\frac{\abs{\hat{Y_1}(l')}^2}{\left(\sum_{j=1}^3\rho^j_{n_j}z_j\right)^2}.\label{midsquare}
\end{align}
It follows from \eqref{midj123} and \eqref{midsquare} that
\begin{align}
\sum_{\substack{0\leq n_i\leq q_i-1\\ i=1,2,3}}\sum_{\substack{l'\in W\\(l_1',l_2',l_3')\neq \mathbf{0}}}&\frac{\abs{\hat{V_1}(l')}^2}{\left(\sum_{j=1}^3\rho^j_{n_j}z_j\right)\left(\sum_{j=1}^3\rho^j_{n_j+l_j'}z_j\right)}\nonumber\\
&=\sum_{\substack{0\leq n_i\leq q_i-1\\ i=1,2,3}}\sum_{\substack{l'\in W\\(l_1',l_2',l_3')\neq \mathbf{0}}}\frac{\abs{\hat{Y_1}(l')}^2}{\left(\sum_{j=1}^3\rho^j_{n_j}z_j\right)\left(\sum_{j=1}^3\rho^j_{n_j+l_j'}z_j\right)}.\label{midlneq0}
\end{align}

Let $l_1\in \{1,2,\cdots,q_1-1\}, l_2\in\{1,2,\cdots,q_2-1\}$ and $l_3\in\{0,1,\cdots,q_3-1\}$ be fixed. By Lemma \ref{lemcoprimeliucpam}, for any $l_j'\in \{0,1,\cdots,q_j-1\},j=1,2,3$ with $(l_1',l_2',l_3')\neq \mathbf{0}$ and $(l_1',l_2',l_3')\neq (l_1,l_2,l_3)$, one has that
\begin{align}
\begin{vmatrix}
1 & 1 & 1 \\
\\
\rho^1_{l_1} & \rho^2_{l_2} &\rho^3_{l_3}\\
\\
\rho^1_{l_1'} & \rho^2_{l_2'} &\rho^3_{l_3'}
\end{vmatrix}\neq 0.\nonumber
\end{align}
Therefore, there exists $\hat{z}^{\#}=(z_1^{\#},z_2^{\#},z_3^{\#})$ such that
\begin{align}
z_1^{\#}+z_2^{\#}+z_3^{\#}=\rho^1_{l_1}z_1^{\#}+\rho^2_{l_2}z_2^{\#}+\rho^3_{l_3}z_3^{\#}=0,\label{planesfermi0}
\end{align}
and for any $(l_1',l_2',l_3')\neq \mathbf{0}$ and $(l_1',l_2',l_3')\neq (l_1,l_2,l_3)$ (modulo periodicity),
\begin{align}
\rho^1_{l_1'}z_1^{\#}+\rho^2_{l_2'}z_2^{\#}+\rho^3_{l_3'}z_3^{\#}\neq0.\label{planesfermineq0}
\end{align}
Let $z_2=z_2^{\#},z_3=z_3^{\#}$ and $z_1\to z_1^{\#}$, by \eqref{midlneq0}, \eqref{planesfermi0} and \eqref{planesfermineq0} we have
\begin{align}
\Bigg(\sum_{\substack{l'\in W\\ (l_1',l_2',l_3')=(l_1,l_2,l_3)}}&+\sum_{\substack{l'\in W\\ (l_1',l_2',l_3')=(q_1-l_1,q_2-l_2,q_3-l_3)}}\Bigg)\frac{\abs{\hat{V}_1(l')}^2}{\rho^1_{l_1}(z_1-z_1^{\#})^2}+\frac{O(1)}{\abs{z_1-z_1^{\#}}}\nonumber\\
=&\Bigg(\sum_{\substack{l'\in W\\ (l_1',l_2',l_3')=(l_1,l_2,l_3)}}+\sum_{\substack{l'\in W\\ (l_1',l_2',l_3')=(q_1-l_1,q_2-l_2,q_3-l_3)}}\Bigg)\frac{\abs{\hat{Y}_1(l')}^2}{\rho^1_{l_1}(z_1-z_1^{\#})^2}+\frac{O(1)}{\abs{z_1-z_1^{\#}}},\nonumber
\end{align}
from which we obtain that for any $l_1\in \{1,2,\cdots,q_1-1\}$, $l_2\in\{1,2,\cdots,q_2-1\}$ and any $l_3\in\{0,1,\cdots,q_3-1\}$,
\begin{align}
\Bigg(\sum_{\substack{l'\in W\\ (l_1',l_2',l_3')=(l_1,l_2,l_3)}}&+\sum_{\substack{l'\in W\\ (l_1',l_2',l_3')=(q_1-l_1,q_2-l_2,q_3-l_3)}}\Bigg)\abs{\hat{V}_1(l')}^2\nonumber\\
=\Bigg(&\sum_{\substack{l'\in W\\ (l_1',l_2',l_3')=(l_1,l_2,l_3)}}+\sum_{\substack{l'\in W\\ (l_1',l_2',l_3')=(q_1-l_1,q_2-l_2,q_3-l_3)}}\Bigg)\abs{\hat{Y}_1(l')}^2.\label{sumv1y1finalfermi}
\end{align}

By \eqref{v1vy1y} and Lemma \ref{stseparablefermi}, for any $l'=(l'_1,l'_2,\cdots,l'_d)\in W$ with $l'_1\neq 0, l'_2\neq 0$, one has \begin{align}
\hat{V_1}(l')=0.\nonumber
\end{align}
Since $l_1,l_2,q_1-l_1,q_2-l_2$ are both non-zero, if follows from  \eqref{sumv1y1finalfermi} that
\begin{align}
\sum_{\substack{l'\in W\\ (l_1',l_2',l_3')=(l_1,l_2,l_3)}}\abs{\hat{Y}_1(l')}^2=0.\nonumber
\end{align}
Since $l_3$ is arbitrary, we obtain that
$$\sum_{\substack{l'\in W\\ (l_1',l_2')=(l_1,l_2)}}\abs{\hat{Y}_1(l')}^2=0.$$
Thus $\hat{Y}_1(l)=0$ for any $l=(l_1,l_2,\cdots,l_d)\in W$ with $l_1\neq 0, l_2\neq 0$. It follows from Lemma \ref{stseparablefermi} that $Y=Y_1+\lambda_2$ is $(1,2)$-separable.
\end{proof}

Now we state our main results, whose proofs follow directly from Theorem \ref{theoremoplussp}, Theorem \ref{d1d2drnewsepafermith} and Theorem \ref{thmferst}.

\begin{theorem}\label{theofer1}
Let $\sum_{j=1}^rd_j=d\geq 3$ with $r\geq 2,d_j\geq1,\ j=1,2,\cdots,r$. Assume that $V$ and $Y$ are real $\Gamma$-periodic functions on $\Z^d$, and $V$ is $(d_1,d_2,\cdots,d_r)$-separable. Suppose further that for any $1\leq i<m\leq r$ and any $s\in \{1+\sum_{j=0}^{i-1}d_j,2+\sum_{j=0}^{i-1}d_j,\cdots,\sum_{j=0}^{i}d_j\}$, $t\in \{1+\sum_{j=0}^{m-1}d_j,2+\sum_{j=0}^{m-1}d_j,\cdots,\sum_{j=0}^{m}d_j\}$, {there exist $\lambda_1=\lambda_1(s,t)\in\C,\lambda_2=\lambda_2(s,t)\in\C$ and $S=S(s,t)\subset \{1,2,\cdots,d\}$ }with $s,t\in S$ and  $\# S\geq 3$, such that $V$ and $Y$ {are generalized partially Fermi isospectral}. Then $Y$ is also $(d_1,d_2,\cdots,d_r)$-separable.
\end{theorem}

{\begin{theorem}
	Let $\sum_{j=1}^rd_j=d\geq 3$ with $r\geq 3,d_j\geq1,\ j=1,2,\cdots,r$. Assume that $V$ and $Y$ are real $\Gamma$-periodic functions on $\Z^d$, and $V$ is $(d_1,d_2,\cdots,d_{r-1})\oplus d_r$-separable. Suppose further that for any $1\leq i<m\leq r-1$ and any $s\in \{1+\sum_{j=0}^{i-1}d_j,2+\sum_{j=0}^{i-1}d_j,\cdots,\sum_{j=0}^{i}d_j\}$, $t\in \{1+\sum_{j=0}^{m-1}d_j,2+\sum_{j=0}^{m-1}d_j,\cdots,\sum_{j=0}^{m}d_j\}$, there exist $\lambda_1=\lambda_1(s,t)\in\C,\lambda_2=\lambda_2(s,t)\in\C$ and $S=S(s,t)\subset \{1,2,\cdots,d\}$ with $s,t\in S$ and  $\# S\geq 3$, such that $V$ and $Y$ are generalized partially Fermi isospectral. Then $Y$ is also $(d_1,d_2,\cdots,d_{r-1})\oplus d_r$-separable.
\end{theorem}}

\begin{corollary}\label{maincorfermi}
Let $V$ and $Y$ be real $\Gamma$-periodic functions on $\Z^d$ with $d\geq 3$, and $V$ is $(1,1,\cdots,1)$-separable. Suppose that for any $1\leq s<t\leq d$, there exist {$\lambda_1=\lambda_1(s,t)\in\C,\lambda_2=\lambda_2(s,t)\in\C$ and $S=S(s,t)\subset \{1,2,\cdots,d\}$} with $s,t\in S$ and $\# S\geq 3$, such that $V$ and $Y$ are {generalized partially Fermi isospectral}. Then $Y$ is also $(1,1,\cdots,1)$-separable.
\end{corollary}

We remark that to prove that two $\Gamma$-periodic potentials have the same $(d_1,d_2,\cdots,d_r)$-separability, we require less information in Theorem \ref{theofer1} and Corollary \ref{maincorfermi} {than those in the existed results}. For example, if $d_1=d_2=\cdots=d_r=\frac{d}{r}$, in some sense, we only need that at most $\binom{r}{2}\times \left(\frac{d}{r}\right)^2$
projections of the Fermi variety on $3$ - dimensional subspace are the same, instead of the coincidence of whole $d$-dimensional Fermi varieties.

Finally in this section, we study the isospectrality of the components of two generalized Fermi isospectral potentials.
Recall that $V$ and $Y$ are Floquet isospectral if $P_V(k,\lambda)=P_Y(k,\lambda)$ for any $k\in\R^d$ and any $\lambda\in\C$.

{\begin{theorem}\label{maintheorem325fer}
Let $d=\sum_{j=1}^rd_j$ with $r\geq 3,d_j\geq 1,\ j=1,2,\cdots,r$, and $d-d_j-d_r\geq 2$, $j=1,2,\cdots,r-1$.
Assume that $V,Y$ are complex $\Gamma$-periodic functions on $\mathbb{Z}^d$, and there exist $\lambda_1,\lambda_2\in\C$ such that $V$ and $Y$ are generalized Fermi isospectral, and
\[V(n)=\sum_{j=1}^{r-1} V_j(\tilde{n}_j,\tilde{n}_r),\quad Y(n)=\sum_{j=1}^{r-1} Y_j(\tilde{n}_j,\tilde{n}_r)\] are both $(d_1,d_2,\cdots,d_{r-1})\oplus d_r$-separable. Then there exist functions $U_{j;1}(\tilde{n}_r)$ and $U_{j;2}(\tilde{n}_r)$, such that $V_j(\tilde{n}_{j},\tilde{n}_r)+U_{j;1}(\tilde{n}_r)$ and $Y_j(\tilde{n}_{j},\tilde{n}_r)+U_{j;2}(\tilde{n}_r)$ are Floquet isospectral, $j=1,2,\cdots,r-1$.
\end{theorem}}

\begin{proof}We only prove the result for $j=1$. Let $n=(n_1,n_2,\cdots,n_d)=(\tilde{n}_1,\tilde{n}_2,\cdots,\tilde{n}_r)$.
Define $U_1:\mathbb{Z}^{d_r}\to\mathbb{C}$ and $U_2:\mathbb{Z}^{d_r}\to\mathbb{C}$ as
\[U_1(\tilde{n}_r)=\sum_{j=2}^{r-1}\frac{1}{\tilde{q}_j}\sum_{\tilde{n}_j\in W_j}V_j(\tilde{n}_j,\tilde{n}_r)+u_1,\quad
U_2(\tilde{n}_r)=\sum_{j=2}^{r-1}\frac{1}{\tilde{q}_j}\sum_{\tilde{n}_j\in W_j}Y_j(\tilde{n}_j,\tilde{n}_r)+u_2,\]
where $u_1\in\C,u_2\in\C$ are some constants will be defined later,
$\tilde{q}_j=\prod_{i=1}^{d_j}q_{i+\sum_{m=0}^{j-1}d_m}$ and
\begin{align}
W_j=\Big\{\tilde{n}_j\in \mathbb{Z}^{d_j}:n=(\tilde{n}_1,\cdots,\tilde{n}_j,\cdots,\tilde{n}_r)=(\mathbf{0},\cdots,\mathbf{0},\tilde{n}_j,\mathbf{0},\cdots,\mathbf{0})\in W\Big\}.\nonumber
\end{align}
Next we prove that $\tilde{V}_1(\tilde{n}_1,\tilde{n}_r)=V_1(\tilde{n}_1,\tilde{n}_r)+U_1(\tilde{n}_r)$ and $\tilde{Y}_1(\tilde{n}_1,\tilde{n}_r)=Y_1(\tilde{n}_1,\tilde{n}_r)+U_2(\tilde{n}_r)$ are Floquet isospectral with proper $u_1,u_2$ being such that
\begin{align}
\hat{\tilde{V}}_1(\mathbf{0},\mathbf{0})=\hat{\tilde{Y}}_1(\mathbf{0},\mathbf{0})=0.\label{mathbf00vy}
\end{align}

For any $\hat{l}=(l_1,\cdots,l_{d_1},l_{d-d_r+1},\cdots,l_d)=(\tilde{l}_1,\tilde{l}_r)$, one has
\begin{align}
\hat{\tilde{V}}_1(\tilde{l}_1,\tilde{l}_r)=&\frac{1}{\tilde{q}_1\tilde{q}_r}\sum_{\tilde{n}_1\in W_1,\tilde{n}_r\in W_r}\left(V_1(\tilde{n}_1,\tilde{n}_r)+U_1(\tilde{n}_r)\right)e^{-2\pi\mathbf{i}\left(\sum_{i=1}^{d_1}+\sum_{i=d-d_r+1}^d\right)\frac{n_il_i}{q_i}}\nonumber\\
=&\begin{cases}
\hat{V}_1(\tilde{l}_1,\tilde{l}_r),&\mathrm{if}\ \tilde{l}_1\neq\mathbf{0},\\
\hat{V}_1(\mathbf{0},\tilde{l}_r)+\hat{U}_1(\tilde{l}_r),&\mathrm{if}\ \tilde{l}_1=\mathbf{0}.
\end{cases}\nonumber
\end{align}
For any $l=(l_1,\cdots,l_{d_1},0\cdots,0,l_{d-d_r+1},\cdots,l_d)=(\tilde{l}_1,\mathbf{0},\cdots,\mathbf{0},\tilde{l}_r)\in W,$
direct computations show that
\begin{align}
\hat{V}(l)=\begin{cases}
\hat{V}_1(\tilde{l}_1,\tilde{l}_r),&\mathrm{if}\ \tilde{l}_1\neq\mathbf{0},\\
\hat{V}_1(\mathbf{0},\tilde{l}_r)+\hat{U}_1(\tilde{l}_r)-u_1\delta_{(l;\mathbf{0})},&\mathrm{if}\ \tilde{l}_1=\mathbf{0}.
\end{cases}\nonumber
\end{align}
{Therefore, for any non-zero $l=(\tilde{l}_1,\mathbf{0},\cdots,\mathbf{0},\tilde{l}_r)\in W$,
\begin{align}
\hat{V}(l)=\hat{\tilde{V}}_1(\tilde{l}_1,\tilde{l}_r).\label{vhatequalvtildehat}
\end{align}
In a similar way,
$\hat{Y}(l)=\hat{\tilde{Y}}_1(\tilde{l}_1,\tilde{l}_r)$ for any $l=(\tilde{l}_1,\mathbf{0},\cdots,\mathbf{0},\tilde{l}_r)\in W$ with $l\neq \mathbf{0}$.}

Note that $d-d_1-d_r\geq 2$.  For any $\lambda\in\C$ and any $\hat{z}=(z_{d_1+2},z_{d_1+3},\cdots,z_{d-d_r})\in\C^{d-d_1-d_r-1}$, denote by $z_{d_1+1}=z_{d_1+1}([V]-\lambda_1,\lambda,\hat{z})$ being the unique solution of
\begin{align}
z_{d_1+1}+\frac{1}{z_{d_1+1}}+\sum_{j=d_1+2}^{d-d_r}\left(z_j+\frac{1}{z_j}\right)+[V]-\lambda_1=-\lambda,\label{zd1lambda}
\end{align}
with $\abs{z_{d_1+1}}\to\infty$ if
$\abs{\sum_{j=d_1+2}^{d-d_r}z_j}\to\infty$
and $\abs{z_j}\to\infty$, $j=d_1+2,d_1+3,\cdots,d-d_r$.

Then for any $n=(n_1,\cdots,n_d)=(\tilde{n}_1,\mathbf{0},\cdots,\mathbf{0},\tilde{n}_r)\in W$,
\begin{align}
[V]-\lambda_1+&\sum_{j=1}^d\Big(\rho^j_{n_j}z_j+\frac{1}{\rho^j_{n_j}z_j}\Big)\nonumber\\
=&-\lambda+\sum_{j=1}^{d_1}\Big(\rho^j_{n_j}z_j+\frac{1}{\rho^j_{n_j}z_j}\Big)+\sum_{j=d-d_r+1}^{d}
\Big(\rho^j_{n_j}z_j+\frac{1}{\rho^j_{n_j}z_j}\Big).\label{diagA}
\end{align}
For any $n=(n_1,\cdots,n_d)=(\tilde{n}_1,\tilde{n}_2,\cdots,\tilde{n}_r)$ with $(\tilde{n}_2,\tilde{n}_3,\cdots,\tilde{n}_{r-1})\neq \mathbf{0}$, as $\abs{\sum_{j=d_1+2}^{d-d_r}z_j}\to\infty$
and $\abs{z_j}\to\infty$, $j=d_1+2,d_1+3,\cdots,d-d_r$, one has
\begin{align}
[V]-\lambda_1&+\sum_{j=1}^d\Big(\rho^j_{n_j}z_j+\frac{1}{\rho^j_{n_j}z_j}\Big)\nonumber\\
=&\left(\sum_{j=1}^{d_1}+\sum_{j=d-d_r+1}^{d}\right)\left(\rho^j_{n_j}z_j+\frac{1}{\rho^j_{n_j}z_j}\right)+\sum_{j={d_1+2}}^{d-d_r}\left(\rho^j_{n_j}-\rho^{d_1+1}_{n_{d_1+1}}\right)z_j+O(1).\label{diagBnonzero}
\end{align}
Since $q_1,q_2,\cdots,q_d$ are pairwise coprime,
\begin{align}
\sum_{j={d_1+2}}^{d-d_r}\left(\rho^j_{n_j}-\rho^{d_1+1}_{n_{d_1+1}}\right)z_j\nonumber
\end{align}
is not identically equal to zero.
Therefore, by \eqref{diagA}, \eqref{diagBnonzero} and Lemma \ref{lemunitary}, for any $\lambda\in\C$ and any $z=(z_1,z_2,\cdots,z_d)=(\tilde{z}_1,z_{d_1+1},\hat{z},\tilde{z}_r)\in\C^d$ with $z_{d_1+1}=z_{d_1+1}([V]-\lambda_1,\lambda,\hat{z})$,
as $\abs{\sum_{j=d_1+2}^{d-d_r}z_j}\to\infty$
and $\abs{z_j}\to\infty$, $j=d_1+2,d_1+3,\cdots,d-d_r$,
\begin{align}
\tilde{\mathcal{P}}_V(z,\lambda_1)=&\tilde{\mathcal{P}}_V(\tilde{z}_1,z_{d_1+1}([V]-\lambda_1,\lambda,\hat{z}),\hat{z},\tilde{z}_r,\lambda_1)\label{pvz1fer}\\
=&\det(A+B-\lambda I)\times\prod_{\substack{n\in W\\(\tilde{n}_2,\tilde{n}_3,\cdots,\tilde{n}_{r-1})\neq \mathbf{0}}}\sum_{j={d_1+2}}^{d-d_r}\left(\rho^j_{n_j}-\rho^{d_1+1}_{n_{d_1+1}}\right)z_j\nonumber\\
&+O\left(\sum_{j=d_1+2}^{d-d_r}\abs{z_j}^{Q-\tilde{q}_1\tilde{q}_r-1}\right),\nonumber
\end{align}
where
\begin{align}
\prod_{\substack{n\in W\\(\tilde{n}_2,\tilde{n}_3,\cdots,\tilde{n}_{r-1})\neq \mathbf{0}}}\sum_{j={d_1+2}}^{d-d_r}\left(\rho^j_{n_j}-\rho^{d_1+1}_{n_{d_1+1}}\right)z_j\nonumber
\end{align}
consists of the terms $\prod_{j=d_1+2}^{d-d_r}z_j^{a_j}$ with $\sum_{j=d_1+2}^{d-d_r}a_j=Q-\tilde{q}_1\tilde{q}_r$,
 $A$ and $B$ are $(\tilde{q}_1\tilde{q}_r)\times (\tilde{q}_1\tilde{q}_r)$ matrices, $A$ is a diagonal matrix with entries
\begin{align}
A((\tilde{n}_1,\tilde{n}_r);(\tilde{n}_1',\tilde{n}_r'))=\left(\sum_{j=1}^{d_1}+\sum_{j=d-d_r+1}^{d}\right)\left(\rho^j_{n_j}z_j+\frac{1}{\rho^j_{n_j}z_j}\right)\delta_{((\tilde{n}_1,\tilde{n}_r);(\tilde{n}_1',\tilde{n}_r'))},\nonumber
\end{align}
and for any $(\tilde{n}_1,\tilde{n}_r)\neq (\tilde{n}_1',\tilde{n}_r')$,  $B((\tilde{n}_1,\tilde{n}_r);(\tilde{n}_1',\tilde{n}_r'))=\hat{V}(\tilde{n}_1-\tilde{n}_1',\mathbf{0},\cdots,\mathbf{0},\tilde{n}_r-\tilde{n}_r')$, and $B((\tilde{n}_1,\tilde{n}_r);(\tilde{n}_1,\tilde{n}_r))=0$.
By \eqref{mathbf00vy} and \eqref{vhatequalvtildehat}, we obtain that
\begin{align}
B((\tilde{n}_1,\tilde{n}_r);(\tilde{n}_1',\tilde{n}_r'))=\hat{\tilde{V}}_1(\tilde{n}_1-\tilde{n}_1',\tilde{n}_r-\tilde{n}_r').\nonumber
\end{align}
Let $z'=(z_1,\cdots,z_{d_1},z_{d-d_r+1},\cdots,z_d)$ and by Lemma \ref{lemunitary}, we conclude  that
\begin{align}
\det(A+B-\lambda I)=\tilde{\mathcal{P}}_{\tilde{V}_1}(z',\lambda).\label{detabpv1}
\end{align}

Similarly, one has
for any $\lambda\in\C$ and any $z=(z_1,z_2,\cdots,z_d)=(\tilde{z}_1,z_{d_1+1},\hat{z},\tilde{z}_r)\in\C^d$ with $z_{d_1+1}=z_{d_1+1}([Y]-\lambda_2,\lambda,\hat{z})$,
as $\abs{\sum_{j=d_1+2}^{d-d_r}z_j}\to\infty$
and $\abs{z_j}\to\infty$, $j=d_1+2,d_1+3,\cdots,d-d_r$,
\begin{align}
\tilde{\mathcal{P}}_Y(z,\lambda_2)=&\tilde{\mathcal{P}}_Y(\tilde{z}_1,z_{d_1+1}([Y]-\lambda_2,\lambda,\hat{z}),\hat{z},\tilde{z}_r,\lambda_2)\label{pyz1fer}\\
=&\det(A'+B'-\lambda I)\times\prod_{\substack{n\in W\\(\tilde{n}_2,\tilde{n}_3,\cdots,\tilde{n}_{r-1})\neq \mathbf{0}}}\sum_{j={d_1+2}}^{d-d_r}\left(\rho^j_{n_j}-\rho^{d_1+1}_{n_{d_1+1}}\right)z_j\nonumber\\
&+O\left(\sum_{j=d_1+2}^{d-d_r}\abs{z_j}^{Q-\tilde{q}_1\tilde{q}_r-1}\right)\nonumber
\end{align}
with
\begin{align}
\det(A'+B'-\lambda I)=\tilde{\mathcal{P}}_{\tilde{Y}_1}(z',\lambda).\label{detabpy1}
\end{align}

By Theorem \ref{v-yequlambda1-2}, $[V]-\lambda_1=[Y]-\lambda_2$. Then by \eqref{zd1lambda} there holds that $z_{d_1+1}([V]-\lambda_1,\lambda,\hat{z})=z_{d_1+1}([Y]-\lambda_2,\lambda,\hat{z})$.
Thus, for any $\lambda\in\C$ and any $z=(z_1,z_2,\cdots,z_d)\in\C^d$ with $z_{d_1+1}=z_{d_1+1}([V]-\lambda_1,\lambda,\hat{z})=z_{d_1+1}([Y]-\lambda_2,\lambda,\hat{z})$,
as $\abs{\sum_{j=d_1+2}^{d-d_r}z_j}\to\infty$
and $\abs{z_j}\to\infty$, $j=d_1+2,d_1+3,\cdots,d-d_r$, by \eqref{pvz1fer}-\eqref{detabpy1}, one has $\tilde{\mathcal{P}}_{\tilde{V}_1}(z',\lambda)=\tilde{\mathcal{P}}_{\tilde{Y}_1}(z',\lambda)$, from which the result follows.
\end{proof}

To finish this section, we prove the following Ambarzumian-type result.

\begin{theorem}
Let $V$ be a real $\Gamma$-periodic function on $\Z^d$ and $Y_0\in\R$ be a constant. Suppose that for any $1\leq s<t\leq d$, there exist $\lambda_1=\lambda_1(s,t)\in\C,\lambda_2=\lambda_2(s,t)\in\C$ and $S=S(s,t)\subset \{1,2,\cdots,d\}$ with $s,t\in S$ and $\# S\geq 3$, such that $V$ and the constant potential $Y=Y_0$ are generalized partially Fermi isospectral. Then $V$ is also a constant function and $V\equiv[V]=Y_0+\lambda_1(s,t)-\lambda_2(s,t)$.
\end{theorem}
\begin{proof}
By Theorem \ref{v-yequlambda1-2}, one has that $Y_0+\lambda_1(s,t)-\lambda_2(s,t)\equiv [V].$
To prove the result, we only need to show that $\hat{V}(l)=0$ for any non-zero $l\in W$.
By Corollary \ref{maincorfermi}, $V$ is $(1,1,\cdots,1)$-separable. For any non-zero $l=(l_1,l_2,\cdots,l_d)\in W$, if there are more than one $l_j\neq 0$, $j=1,2,\cdots,d$,
by Lemma \ref{stseparablefermi} and Theorem \ref{d1d2drnewsepafermith}, $\hat{V}(l)=0$. Therefore, we only need to prove that $\hat{V}(l)=0$ for any non-zero $l\in W$ with only one $l_j\neq 0$, say for example $l=(l_1,0,\cdots,0)$ with $l_1\in\{1,2,\cdots,q_1-1\}$.

By the assumption, there exist $\lambda_1,\lambda_2\in \C$ and $S\subset\{1,2,\cdots,d\}$ with $\{1,2\}\subset S$ and $\# S\geq 3$, such that $V$ and $Y=Y_0$ are generalized partially Fermi isospectral. Without loss of generality, assume that $S=\{1,2,3\}$. Then we have \eqref{midj123}. Let $z_1=0$, for all possible $(z_2,z_3)\in\C^2$, one has
\begin{align}
\sum_{\substack{0\leq n_i\leq q_i-1\\ i=2,3}}&\sum_{l'\in W}\frac{\abs{\hat{V_1}(l')}^2}{\left(\sum_{j=2}^3\rho^j_{n_j}z_j\right)\left(\sum_{j=2}^3\rho^j_{n_j+l_j'}z_j\right)}\nonumber\\
&=\sum_{\substack{0\leq n_i\leq q_i-1\\ i=2,3}}\sum_{l'\in W}\frac{\abs{\hat{Y_1}(l')}^2}{\left(\sum_{j=2}^3\rho^j_{n_j}z_j\right)\left(\sum_{j=2}^3\rho^j_{n_j+l_j'}z_j\right)},\label{fineefer}
\end{align}
where $V_1(n)=V(n)-\lambda_1$ and $Y_1(n)\equiv Y_0-\lambda_2$. Let $z_0^*\in\C$ be non-zero, since $q_2$ and $q_3$ are coprime, we have $\rho^2_{n_2}z_0^*-\rho^3_{n_3}z_0^*\neq0$ for any $(n_2,n_3)\neq (0,0)$ (module periodicity). Let $z_2=z_0^*$ and $z_3\to-z_0^*$, one obtains from \eqref{fineefer} that
\begin{align}
\sum_{\substack{l'\in W\\ l_2'=l_3'=0}}\abs{\hat{V_1}(l')}^2=\sum_{\substack{l'\in W\\ l_2'=l_3'=0}}\abs{\hat{Y_1}(l')}^2.\label{fisumfer}
\end{align}
Since $\hat{Y}_1(l')=(Y_0-\lambda_2)\delta_{(l';\mathbf{0})}$ and $\hat{V}_1(\mathbf{0})=[V]-\lambda_1=Y_0-\lambda_2$, it follows from \eqref{fisumfer} that $\hat{V}_1(l)=0$ for any $l=(l_1,0,\cdots,0)\in W$ with $l_1\neq 0$. Due to \eqref{v1vy1y}, the result holds.
\end{proof}
\begin{corollary}
	Let $V$ be a real $\Gamma$-periodic function on $\Z^d$. Suppose that for any $1\leq s<t\leq d$, there exist $\lambda_1(s,t)=\lambda_2(s,t)\in\C$ and $S=S(s,t)\subset \{1,2,\cdots,d\}$ with $s,t\in S$ and $\# S\geq 3$, such that $V$ and the zero potential are partially Fermi isospectral. Then $V\equiv0$.
\end{corollary}

\vspace{2mm}\noindent \textbf{Data availability}
No data was used for the research described in the article.

\vspace{2mm}\noindent \textbf{Conflict of interest} The authors report no conflict of interest.

\bibliographystyle{amsplain}

\end{document}